\documentclass[12pt]{article}

\usepackage[english]{babel}
\usepackage[T1]{fontenc}
\usepackage[cp1250]{inputenc}
\usepackage{bbold}
\usepackage{url}
\usepackage{graphicx}
\usepackage{longtable}
\usepackage{amsthm}
\usepackage{tikz}
\usepackage{authblk}
\usepackage{subcaption}
\usetikzlibrary{calc}

\newtheorem{theorem}{Theorem}

\newtheorem{claim}[theorem]{Claim}

\newtheorem{lemma}[theorem]{Lemma}

\newtheorem{remark}[theorem]{Remark}

\newtheorem*{mytheorem1}{Theorem \ref{th:gen}}
\newtheorem*{mytheorem2}{Theorem \ref{th:frac}}

\title{Clique number of the square of a line graph}%
\author[1]{Ma\l gorzata \'{S}leszy\'{n}ska-Nowak\thanks{m.sleszynska@mini.pw.edu.pl}}
\affil[1]{Faculty of Mathematics and Information Science, Warsaw University of Technology, Warsaw, Poland}

\begin{document}
\maketitle

\begin{abstract}

An \emph{edge coloring} of a graph $G$ is strong if each color class is an induced matching of $G$. The \emph{strong chromatic index} of $G$, denoted by $\chi _{s}^{\prime }(G)$, is the minimum number of colors for which $G$ has a strong edge coloring. The strong chromatic index of $G$ is equal to the chromatic number of the square of the line graph of $G$. The chromatic number of the square of the line graph of $G$ is greater than or equal to the clique number of the square of the line graph of $G$, denoted by $\omega(L)$. 

In this note we prove that $\omega(L) \le 1.5 \Delta_{G}^2$ for every graph $G$. Our result allows to calculate an upper bound for the fractional strong chromatic index of $G$, denoted by $\chi_{fs}^\prime(G)$. We prove that $\chi_{fs}^{\prime}(G) \le 1.75 \Delta_G^2$ for every graph $G$.

\end{abstract}

\section{Introduction}

A \emph{strong edge coloring} of a graph $G$ is an edge coloring in which every color class is an induced matching, that is, any two vertices that belong to distinct edges of the same color are not adjacent (in particular, no two edges of the same color intersect). The strong chromatic index of $G$, denoted by $\chi_{s}^{\prime}(G)$, is the minimum number of colors in a strong edge coloring. 

The concept of the strong edge coloring was introduced around 1985 by Erd\H{o}s and Ne\v{s}et\v{r}il \cite{ErdosNesetril}. They conjectured that for every graph $G$, with maximum degree $\Delta_G$, $\chi_{s}^{'}(G) \le \frac{5}{4}\Delta_G^2$. The example of graph obtained from the cycle of length five by replacing each vertex by an independent set of size $\frac{\Delta}{2}$ shows that this bound, if true, is tight. 

The trivial bound for the strong chromatic index of $G$ is $2\Delta_G ^{2}-2\Delta_G +1$. Molloy and Reed \cite{MoloyReed} proved that $\chi_{s}^{'}(G) \le (2-\epsilon )\Delta_G ^{2}$ for $\Delta_G $ sufficiently large, where $\epsilon $ is a small constant around $\frac{1}{50}$. Recently Bruhn and Joos \cite{BruhJoos} improved it to the $1.93\Delta_G ^{2}$.

We can look at this problem from a different angle. A \emph{line graph of $G$} is a graph whose each vertex represents an edge of $G$ and two vertices are adjacent if and only if their corresponding edges are incident in $G$. A \emph{square of a graph $H$} is a graph with the same set of vertices as $H$, in which two vertices are adjacent when their distance in $H$ is at most 2. The strong chromatic index of the graph $G$ is equal to the chromatic number of the square of the line graph of $G$ (say $L$). The chromatic number of $L$ is greater than or equal to the clique number of $L$ (denoted by $\omega(L)$), so finding bounds for $\omega(L)$ is also an interesting problem. It is not known if the clique number of $L$ is bounded by $\frac{5}{4}\Delta_G^2$ \cite{SixProblems}. Chung et al. \cite{ChungEtAl} proved that if $L$ is a clique then $G$ has at most $\frac{5}{4}\Delta_G^2$ edges. Faudree et al. \cite{FaudreeEtAl2} proved that if $G$ is a bipartite graph then the clique number of $L$ is at most $\Delta_G^2$. The example of $K_{\Delta,\Delta}$ shows that this bound is tight. Recently Bruhn and Joos \cite{BruhJoos} proved that for $G$ with $\Delta_G \ge 400$ the clique number of $L$ is at most $1.74\Delta_G^2$. We improved this result.

In this note, we focus on two topics related with the strong chromatic index of graphs. The first problem is finding the upper bound of the clique number of the square of the line graph of $G$. In our main theorem we proved an upper bound for the general case:
\begin{mytheorem1}
\label{th:gen}
Let $G$ be a simple graph and $L$ be a square of the line graph of $G$. Then the clique number of $L$ is at most $1.5 \Delta_{G}^2$.
\end{mytheorem1}
We also present (Theorem \ref{th:bip}) a new proof of known bounds for bipartite graphs ($\Delta_G^2$). 
The second subject is a fractional strong chromatic index, denoted by $\chi_{fs}^{\prime}(G)$. Using Theorem \ref{th:gen} we proved an upper for $\chi_{fs}^{\prime}(G)$:

\begin{mytheorem2}
\label{th:frac}
Let $G$ be a simple graph. Then the fractional strong chromatic index of $G$ is at most $1.75 \Delta_G^2$.
\end{mytheorem2}

\section{Upper bound for the clique number of the square of the line graph of $G$}

For any two different edges from a graph $G$ we define \emph{$dist_G(e,f)$} as a number of edges in the shortest path between $e$ and $f$ plus 1, i.e. $dist_G(e, f) = 1 $ iff $e$ and $f$ intersect.

\begin{remark}
Let $G$ be a simple graph, $L$ be a square of the line graph of $G$ and $H$ be a subgraph of $G$, such that for each $e,f \in E(H)$ we have $dist_{G}(e, f) \le 2$. Then $\omega(L) = |E(H)|$.
\end{remark}

\subsection{Bipartite graphs}
For the better understanding of our method, first we will present a new proof of known bounds for bipartite graphs. In the proof of our main theorem, which is more complicated, we will use the same technique.

\begin{theorem}
\label{th:bip}
Let $G$ be a simple bipartite graph and $L$ be a square of the line graph of $G$. Then the clique number of $L$ is at most $\Delta_{G}^2$.
\end{theorem}




\begin{proof}
Let $H$ be a subgraph of $G$, such that for each $e,f \in E(H)$ we have $dist_{G}(e, f) \le 2$. We will show that $H$ has at most $\Delta_{G}^2$ edges.

Consider a vertex $v \in V(H)$ of degree $\Delta_H$. Edges of the graph $H$ can be divided into following sets (see figure \ref{fig:sets_bip}):

\begin{enumerate}
\item $A = \{e \in E(H): v \in e\}$, $|A| = \Delta_H$

$A$ is a set of edges of $H$ which are incident to $v$.


Notice that all other edges of the graph $H$ are in a distance at most 2 to all edges from $A$.

\item $B = \{e \in E(H): e \notin A \wedge \exists_{f \in A} e$ and $f$ are adjacent$\}$, $|B| \le \Delta_H (\Delta_H - 1)$

$B$ is a set of edges of $H$ which are adjacent to edges from $A$ and are not contained in $A$.

\item $C = \{e \in E(H): \exists_{f \in E(G)}(v \in f \wedge f \notin E(H)$  $\wedge$  $e$ and $f$ are adjacent$)\}$, $|C| \le (\Delta_G - \Delta_H)\Delta_H$

$C$ is a set of edges of $H$ which are adjacent to edges of $G - H$ which are incident to $v$.

\item $D = \{e \in E(H): e \notin C \wedge \forall_{f \in A}dist_G(e, f) = 2\}$

$D$ is a set of edges of $H$ which are at the distance 2 to all edges from $A$ and are not contained in $C$.

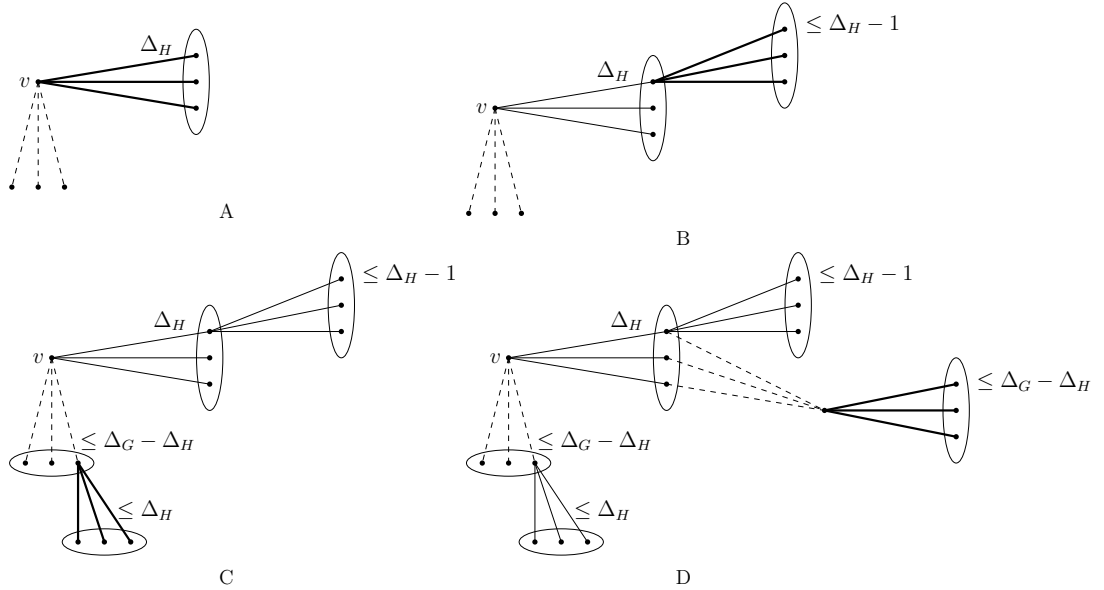
\begin{figure}[th]
\centering
\scalebox{.7}{
\begin{tabular}{cc}
\begin{subfigure}{.5\textwidth}
\begin{tikzpicture}
\coordinate (A) at (0.5,5.5);
\coordinate (B) at (3.5,5);
\coordinate (C) at (3.5,5.5);
\coordinate (D) at (3.5,6);
\coordinate (H) at (0,3.5);
\coordinate (I) at (0.5,3.5);
\coordinate (J) at (1,3.5);

\path[fill=black] (A) circle (0.05);

\path[fill=black] (B) circle (0.05);
\draw[very thick](A) -- (B);
\path[fill=black] (C) circle (0.05);
\draw[very thick](A) -- (C);
\path[fill=black] (D) circle (0.05);
\draw[very thick](A) -- (D);

\path[fill=black] (H) circle (0.05);
\draw[dashed](A) -- (H);
\path[fill=black] (I) circle (0.05);
\draw[dashed](A) -- (I);
\path[fill=black] (J) circle (0.05);
\draw[dashed](A) -- (J);

\draw ($(A)!1!(C)$) ellipse (0.25 and 1);

\node at (A) [left] {$v$};
\node at ($(A)!.75!(D)$) [above] {$\Delta_H$};


\end{tikzpicture}
\caption{A}
\end{subfigure}
&
\begin{subfigure}{.5\textwidth}
\begin{tikzpicture}

\coordinate (A) at (8,5.5);
\coordinate (B) at (11,5);
\coordinate (C) at (11,5.5);
\coordinate (D) at (11,6);
\coordinate (E) at (13.5,6);
\coordinate (F) at (13.5,6.5);
\coordinate (G) at (13.5,7);
\coordinate (H) at (7.5,3.5);
\coordinate (I) at (8,3.5);
\coordinate (J) at (8.5,3.5);

\path[fill=black] (A) circle (0.05);

\path[fill=black] (B) circle (0.05);
\path[draw](A) -- (B);
\path[fill=black] (C) circle (0.05);
\path[draw](A) -- (C);
\path[fill=black] (D) circle (0.05);
\path[draw](A) -- (D);

\path[fill=black] (E) circle (0.05);
\draw[very thick](D) -- (E);
\path[fill=black] (F) circle (0.05);
\draw[very thick](D) -- (F);
\path[fill=black] (G) circle (0.05);
\draw[very thick](D) -- (G);

\path[fill=black] (H) circle (0.05);
\draw[dashed](A) -- (H);
\path[fill=black] (I) circle (0.05);
\draw[dashed](A) -- (I);
\path[fill=black] (J) circle (0.05);
\draw[dashed](A) -- (J);

\draw ($(A)!1!(C)$) ellipse (0.25 and 1);
\draw ($(D)!1!(F)$) ellipse (0.25 and 1);

\node at (A) [left] {$v$};
\node at ($(A)!.75!(D)$) [above] {$\Delta_H$};
\node at ($(D)!1.1!(G)$) [right] {$\le \Delta_H - 1$};

\end{tikzpicture}
\caption{B}
\end{subfigure}
\\
\begin{subfigure}{.5\textwidth}
\begin{tikzpicture}

\coordinate (A) at (0.5,0.5);
\coordinate (B) at (3.5,0);
\coordinate (C) at (3.5,0.5);
\coordinate (D) at (3.5,1);
\coordinate (E) at (6,1);
\coordinate (F) at (6,1.5);
\coordinate (G) at (6,2);
\coordinate (H) at (0,-1.5);
\coordinate (I) at (0.5,-1.5);
\coordinate (J) at (1,-1.5);
\coordinate (K) at (1,-3);
\coordinate (L) at (1.5,-3);
\coordinate (M) at (2,-3);

\path[fill=black] (A) circle (0.05);

\path[fill=black] (B) circle (0.05);
\path[draw](A) -- (B);
\path[fill=black] (C) circle (0.05);
\path[draw](A) -- (C);
\path[fill=black] (D) circle (0.05);
\path[draw](A) -- (D);

\path[fill=black] (E) circle (0.05);
\path[draw](D) -- (E);
\path[fill=black] (F) circle (0.05);
\path[draw](D) -- (F);
\path[fill=black] (G) circle (0.05);
\path[draw](D) -- (G);

\path[fill=black] (H) circle (0.05);
\draw[dashed](A) -- (H);
\path[fill=black] (I) circle (0.05);
\draw[dashed](A) -- (I);
\path[fill=black] (J) circle (0.05);
\draw[dashed](A) -- (J);

\path[fill=black] (K) circle (0.05);
\draw[very thick](J) -- (K);
\path[fill=black] (L) circle (0.05);
\draw[very thick](J) -- (L);
\path[fill=black] (M) circle (0.05);
\draw[very thick](J) -- (M);

\draw ($(A)!1!(C)$) ellipse (0.25 and 1);
\draw ($(D)!1!(F)$) ellipse (0.25 and 1);
\draw ($(A)!1!(I)$) ellipse (0.8 and 0.25);
\draw ($(J)!1!(L)$) ellipse (0.8 and 0.25);

\node at (A) [left] {$v$};
\node at ($(A)!.75!(D)$) [above] {$\Delta_H$};
\node at ($(D)!1.1!(G)$) [right] {$\le \Delta_H - 1$};
\node at ($(A)!.79!(J)$) [right] {$\le \Delta_G - \Delta_H $};
\node at ($(J)!.6!(M)$) [right] {$\le \Delta_H $};

\end{tikzpicture}
\caption{C}
\end{subfigure}
&
\begin{subfigure}{.5\textwidth}
\begin{tikzpicture}

\coordinate (A) at (8,0.5);
\coordinate (B) at (11,0);
\coordinate (C) at (11,0.5);
\coordinate (D) at (11,1);
\coordinate (E) at (13.5,1);
\coordinate (F) at (13.5,1.5);
\coordinate (G) at (13.5,2);
\coordinate (H) at (7.5,-1.5);
\coordinate (I) at (8,-1.5);
\coordinate (J) at (8.5,-1.5);
\coordinate (K) at (8.5,-3);
\coordinate (L) at (9,-3);
\coordinate (M) at (9.5,-3);
\coordinate (N) at (14,-0.5);
\coordinate (O) at (16.5,-1);
\coordinate (P) at (16.5,-0.5);
\coordinate (R) at (16.5,0);

\path[fill=black] (A) circle (0.05);

\path[fill=black] (B) circle (0.05);
\path[draw](A) -- (B);
\path[fill=black] (C) circle (0.05);
\path[draw](A) -- (C);
\path[fill=black] (D) circle (0.05);
\path[draw](A) -- (D);

\path[fill=black] (E) circle (0.05);
\path[draw](D) -- (E);
\path[fill=black] (F) circle (0.05);
\path[draw](D) -- (F);
\path[fill=black] (G) circle (0.05);
\path[draw](D) -- (G);

\path[fill=black] (H) circle (0.05);
\draw[dashed](A) -- (H);
\path[fill=black] (I) circle (0.05);
\draw[dashed](A) -- (I);
\path[fill=black] (J) circle (0.05);
\draw[dashed](A) -- (J);

\path[fill=black] (K) circle (0.05);
\path[draw](J) -- (K);
\path[fill=black] (L) circle (0.05);
\path[draw](J) -- (L);
\path[fill=black] (M) circle (0.05);
\path[draw](J) -- (M);

\path[fill=black] (N) circle (0.05);
\path[fill=black] (O) circle (0.05);
\path[fill=black] (P) circle (0.05);
\path[fill=black] (R) circle (0.05);
\draw[dashed](B) -- (N);
\draw[dashed](C) -- (N);
\draw[dashed](D) -- (N);
\draw[very thick](N) -- (O);
\draw[very thick](N) -- (P);
\draw[very thick](N) -- (R);

\draw ($(A)!1!(C)$) ellipse (0.25 and 1);
\draw ($(D)!1!(F)$) ellipse (0.25 and 1);
\draw ($(A)!1!(I)$) ellipse (0.8 and 0.25);
\draw ($(J)!1!(L)$) ellipse (0.8 and 0.25);
\draw ($(N)!1!(P)$) ellipse (0.25 and 1);

\node at (A) [left] {$v$};
\node at ($(A)!.75!(D)$) [above] {$\Delta_H$};
\node at ($(D)!1.1!(G)$) [right] {$\le \Delta_H - 1$};
\node at ($(A)!.79!(J)$) [right] {$\le \Delta_G - \Delta_H $};
\node at ($(J)!.6!(M)$) [right] {$\le \Delta_H $};
\node at ($(N)!1.1!(R)$) [right] {$\le \Delta_G - \Delta_H $};


\end{tikzpicture}
\caption{D}
\end{subfigure}
\\
\end{tabular}
}
\caption{Sets A, B, C, D}
\label{fig:sets_bip}
\end{figure}

Let $S$ be a subgraph of $H$ induced by D.

We define a super vertex as a vertex which is adjacent in $G$ to all neighbors of $v$ from $H$, and it is not $v$. Because $G$ is bipartite, each edge from $S$ contain exactly one super vertex. We have at most $\Delta_G - 1$ super vertices in $G$. Furthermore each super vertex is incident in $S$ to at most $\Delta_G - \Delta_H$ edges (because there is exactly $\Delta_H$ edges incident to this super vertex in $G$ from neighbors of $v$ from $H$). So the cardinality of $D$ is at most
\begin{center}
$(\Delta_G - 1)(\Delta_G - \Delta_H) $
\end{center}
\end{enumerate}

Now we can sum up the number of edges in $H$
\begin{center}
$|E(H)| \le \Delta_H + \Delta_H(\Delta_H - 1) + (\Delta_G - \Delta_H)\Delta_H + (\Delta_G - 1)(\Delta_G - \Delta_H) =
 \Delta_G^2 - \Delta_G + \Delta_H \le \Delta_G^2$
\end{center}
\end{proof}

\subsection{General case}

In our proof we will use the following lemma.

\begin{lemma}
\label{lem:first}
Let $G$ be a graph with maximum degree $\Delta$, $p$ and $w$ be integers such that $\Delta \le p \le w$ and $\Delta > w - p$. Consider $p$ vertex covers of $G$ (not necessarily different) such that each vertex cover contains at most $w$ vertices.
Moreover assume that for each vertex $v \in V(G)$ we have $deg(v) \le w - a$, where $a$ is a number of vertex covers which contain $v$. Then the graph $G$ has at most $w^{2} - \frac{pw}{2}$ edges.
\end{lemma}

\begin{proof}
Let $v \in V(G)$ be a vertex with degree $\Delta$. The vertex $v$ is contained in at most $w - \Delta$ vertex covers (by assumption), so every neighbor of $v$ is contained in at least $p - w + \Delta$ vertex covers (note that $p - w + \Delta > 0$). Therefore every neighbor of $v$ has degree at most $2w - p - \Delta$.


Let us consider one vertex cover (say $vc$) which does not contain $v$. $vc$ contains exactly $\Delta$ neighbors of $v$ and at most $w - \Delta$ other vertices. Every edge from $G$ has to be incident to at least one vertex from $vc$. Therefore we have
\begin{center}
$|E(G)| \le \Delta (2w - p - \Delta) + (w - \Delta) \Delta = \Delta( 3w - p - 2\Delta)$
\end{center}
and moreover
\begin{center}
$|E(G)| \le w \Delta$.
\end{center}
\begin{figure}[th]
\centering
\includegraphics[width = 7cm]{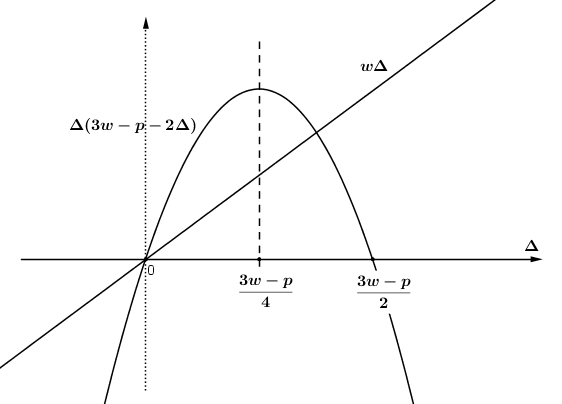}
\caption{Bounding functions}
\label{fig:bounds}
\end{figure}

From these bounds we have the number of edges of $G$ bounded by $w^{2} - \frac{pw}{2}$ (see figure \ref{fig:bounds}, the functions intersect at point $\Delta = \frac{2w - p}{2}$ and for $p \le w $ and $ w \ge 0$ we have $ \frac{3w - p}{4} \le \frac{2w - p}{2} \le \frac{3w - p}{2}$, so the maximum value occurs at the point of intersection).
\end{proof}

\begin{remark}
Lemma \ref{lem:first} is sharp. $G$ has exactly $w^{2} - \frac{pw}{2}$ edges for: $G = 2K_{\Delta, \Delta}$, $w = p = 2\Delta$ and following vertex covers: $\Delta$ vertex covers contain all vertices from the first partition class of both copies of $K_{\Delta, \Delta}$, next $\Delta$ vertex covers contain all vertices from the second partition class of both copies of $K_{\Delta, \Delta}$.
\end{remark}

Now we can prove our main theorem.

\begin{theorem}
\label{th:gen}
Let $G$ be a simple graph and $L$ be a square of the line graph of $G$. Then the clique number of $L$ is at most $1.5 \Delta_{G}^2$.
\end{theorem}




\begin{proof}
Let $H$ be a subgraph of $G$, such that for each $ e,f \in E(H)$ we have $dist_{G}(e, f) \le 2$. We will show that $H$ has at most $1.5 \Delta_{G}^2$ edges. Notice that $0 \le \Delta_H \le \Delta_G$, because $H$ is a subgraph of $G$.

If $\Delta_H < 0.75 \Delta_G$, we have $|E(G)| \le 2 \Delta_G \Delta_H < 1.5 \Delta_G^2$, so there is nothing to prove. If $0.75 \Delta_G \le \Delta_H \le \Delta_G$ the proof is similar to the proof of the theorem \ref{th:bip}:

Consider a vertex $v \in V(H)$ with degree $\Delta_H$. We can divide edges from $H$ into four sets: $A$, $B$, $C$, $D$. Definitions of sets are the same as in the proof of theorem \ref{th:bip}, the cardinality is the same for $A$, $B$ and $C$. Let us look at the set $D$:

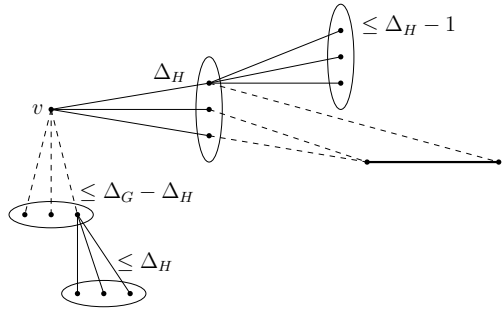
\begin{figure}[th]
\centering
\scalebox{.7}{
\begin{tikzpicture}

\coordinate (A) at (8,0.5);
\coordinate (B) at (11,0);
\coordinate (C) at (11,0.5);
\coordinate (D) at (11,1);
\coordinate (E) at (13.5,1);
\coordinate (F) at (13.5,1.5);
\coordinate (G) at (13.5,2);
\coordinate (H) at (7.5,-1.5);
\coordinate (I) at (8,-1.5);
\coordinate (J) at (8.5,-1.5);
\coordinate (K) at (8.5,-3);
\coordinate (L) at (9,-3);
\coordinate (M) at (9.5,-3);
\coordinate (N) at (14,-0.5);
\coordinate (O) at (16.5,-0.5);

\path[fill=black] (A) circle (0.05);

\path[fill=black] (B) circle (0.05);
\path[draw](A) -- (B);
\path[fill=black] (C) circle (0.05);
\path[draw](A) -- (C);
\path[fill=black] (D) circle (0.05);
\path[draw](A) -- (D);

\path[fill=black] (E) circle (0.05);
\path[draw](D) -- (E);
\path[fill=black] (F) circle (0.05);
\path[draw](D) -- (F);
\path[fill=black] (G) circle (0.05);
\path[draw](D) -- (G);

\path[fill=black] (H) circle (0.05);
\draw[dashed](A) -- (H);
\path[fill=black] (I) circle (0.05);
\draw[dashed](A) -- (I);
\path[fill=black] (J) circle (0.05);
\draw[dashed](A) -- (J);

\path[fill=black] (K) circle (0.05);
\path[draw](J) -- (K);
\path[fill=black] (L) circle (0.05);
\path[draw](J) -- (L);
\path[fill=black] (M) circle (0.05);
\path[draw](J) -- (M);

\path[fill=black] (N) circle (0.05);
\path[fill=black] (O) circle (0.05);
\draw[dashed](B) -- (N);
\draw[dashed](C) -- (N);
\draw[dashed](D) -- (O);
\draw[very thick](N) -- (O);

\draw ($(A)!1!(C)$) ellipse (0.25 and 1);
\draw ($(D)!1!(F)$) ellipse (0.25 and 1);
\draw ($(A)!1!(I)$) ellipse (0.8 and 0.25);
\draw ($(J)!1!(L)$) ellipse (0.8 and 0.25);

\node at (A) [left] {$v$};
\node at ($(A)!.75!(D)$) [above] {$\Delta_H$};
\node at ($(D)!1.1!(G)$) [right] {$\le \Delta_H - 1$};
\node at ($(A)!.79!(J)$) [right] {$\le \Delta_G - \Delta_H $};
\node at ($(J)!.6!(M)$) [right] {$\le \Delta_H $};

\end{tikzpicture}
}
\caption{Sets A, B, C, D}
\label{fig:sets_gen}
\end{figure}

Let $S$ be a subgraph of $H$ induced by D.

\begin{claim}
The graph $S$ has at most $\Delta_G^2 - \frac{\Delta_G \Delta_H}{2}$edges.
\end{claim}

\begin{proof}

Notice that $\Delta_S \le \Delta_H \le \Delta_G$. There are two cases:

\begin{itemize}
\item Case 1:
\begin{center}
$ \Delta_S \le \Delta_G - \Delta_H$
\end{center}
Notice that $\Delta_S \le 0.25 \Delta_G$ (because $ 0.75 \le \Delta_H \le \Delta_G$).

Let us calculate number of edges of $S$. Each neighbor of $v$ from $H$ can be adjacent to at most $\Delta_G$ vertices in $G$, each such vertex can be incident to at most $\Delta_S$ edges in $S$, moreover each edge from $S$ has to contain at least one such vertex. Therefore
\begin{center}
$|E(S)| \le \Delta_G \Delta_S \le 0.25 \Delta_G^2 \le \Delta_G^2 - \frac{\Delta_G \Delta_H}{2}$,
\end{center}
which was to be proved.
\item Case 2:
\begin{center}
$\Delta_S > \Delta_G - \Delta_H$
\end{center}
We use Lemma \ref{lem:first}.

For $ p = \Delta_H$ and $w = \Delta_G$ we have following $p$ vertex covers of $S$: for each neighbor of $v$ from $H$ (say $vs$), all vertices from $S$ which are adjacent in $G$ to $vs$ form a vertex cover of $S$. Each vertex cover has at most $w$ vertices.

Moreover for each vertex $u$ from $S$ we have $deg_S(u) \le \Delta_G - a$, where $a$ is a number of neighbors of $v$ from $H$ which are adjacent in $G$ to $u$ (so a is a number of vertex covers that contain $u$).

Therefore 
\begin{center}
$|E(S)| \le \Delta_G^2 - \frac{\Delta_G \Delta_H}{2}$.
\end{center}
\end{itemize}

\end{proof}
Now we can sum up the number of edges in $H$
\begin{center}
$|E(H)| \le \Delta_H + \Delta_H(\Delta_H - 1) + (\Delta_G - \Delta_H)\Delta_H + \Delta_G^2 - \frac{\Delta_G \Delta_H}{2} = \Delta_G^2 + \frac{\Delta_G \Delta_H}{2} \le 1.5 \Delta_G^2$
\end{center}
\end{proof}

\section{A fractional strong chromatic index}

We will introduce a definition of a fractional strong chromatic index of graph and prove an upper bound of it.

Let $G$ be a simple graph and $M(G)$ be a set of induced matchings in $G$. A \emph{fractional strong edge coloring} of $G$ is a non-negative weighting $w$ of $M(G)$ so that for each $e \in E(G)$ we have $\sum\limits_{m: e \in m} w(m) = 1.$

The weight $\alpha$ of this coloring is defined by:
\begin{center}
$\alpha = \sum\limits_{m \in M(G)} w(m),$
\end{center}
in this case we say that $G$ has a fractional strong $\alpha-$edge coloring.

The \emph{fractional strong chromatic index} of $G$, denoted by $\chi_{fs}^\prime(G)$  is the minimum $\alpha$ for which $G$ has a fractional strong $\alpha$-edge coloring.

\begin{theorem}
\label{th:frac}
Let $G$ be a simple graph. Then the fractional strong chromatic index of $G$ is at most $1.75 \Delta_G^2$.
\end{theorem}

\begin{proof}
Let $L$ be a square of the line graph of $G$. The fractional strong chromatic index of $G$ is equal to the fractional strong chromatic number of $L$.

Molloy and Reed \cite{MoloyReed2} proved that the fractional strong chromatic number of the graph $G$ is at most 
\begin{center}
$\frac{\omega(G) + \Delta_G + 1}{2}$.
\end{center}
From the Theorem \ref{th:bip} we have that $\omega(L) \le 1.5 \Delta_G^2$. The maximum degree of $L$ is at most $2\Delta_G^2 - 2\Delta_G$.

So we have
\begin{center}
$\chi_{fs}^\prime(G) \le \frac{1.5 \Delta_G^2 + 2\Delta_G^2 - 2 \Delta_G+ 1}{2} = 1.75 \Delta_G^2 - \Delta_G + 0.5 \le 1.75 \Delta_G^2$.
\end{center}
\end{proof}

\section{Final remarks}
\subsection{Improving the upper bound}
We feel that it is possible to improve Theorem \ref{th:gen}. Lemma \ref{lem:first} is tight but when we count edges from $S$ we do not use all information about the structure of $S$. We know that $S$ is a clique in the square of the line graph of $G$. Perhaps using this information will give a better bound. 

\subsection{Molloy and Reed conjecture}
Molloy and Reed \cite{MoloyReed2} conjectured that the chromatic number of a graph $G$ is at most $\lceil\frac{\omega(G) + \Delta_G + 1}{2}\rceil$.
If this conjecture is true, our Theorem \ref{th:gen} will give us a new upper bound for the strong chromatic index of $G$: $\chi_s^\prime(G) \le \lceil1.75 \Delta_G^2\rceil$.

\subsection{Degenerate graphs}
A graph $G$ is \emph{$k$-degenerate} if every subgraph of $G$ contains a vertex of degree at most $k$. It is known that the strong chromatic index for degenerate graphs is bounded by a linear function of $k^2\Delta_G$ (see \cite{ChangNarayanan}, \cite{Barrett} and \cite{DGS}). Of course it implies a bound for the clique number of the square of the line graph but it would be nice to know something more. In particular, what is a clique number of the square of the line graph of planar graph?

\subsection{Other problems}
There are many other absorbing conjectures about the strong chromatic index of special classes of graphs. Faudree et al. \cite{FaudreeEtAl} conjectured that $\chi _{s}^{\prime }(G)\leq \Delta_G ^{2}$ when $G$ is a bipartite graph. Brualdi and Quinn \cite{BrualdiQuinn} formulated stronger conjecture: if $G$ is a bipartite graph and $\Delta _{i}$ is the maximum degree of a vertex in the $i$-th partition class then $\chi_{s}^{\prime }(G)\leq \Delta _{1}\Delta _{2}$.

\section{Acknowledgements}
Research was supported by the Polish National Science Center, decision no DEC-2013/11/N/ST1/03199.

\end{document}